\documentclass[reqno,oneside]{amsart}
\usepackage{amsfonts,cite}
\usepackage{amsmath,amsthm,mathrsfs}
\usepackage{amsfonts,amssymb,geometry}
\usepackage{CJK}
\usepackage[all]{xy}

\newtheorem{thm}{Theorem}[section]

\newtheorem{lem}[thm]{Lemma}

\theoremstyle{remark}
\newtheorem{rem}[thm]{Remark}

\numberwithin{equation}{section}

 \def\lz{{\lambda}}

 \def\sz{{\sigma}}
 \def\vz{{\varepsilon}}

\def\Lz{\Lambda}

\def\a{{\mathbf a}}
 \def\b{{\mathbf b}}

\def\({\Bigl(}
\def \){ \Bigr)}

\def\x{{\bf x}}
\def\y{{\bf y}}

\makeatletter
\@namedef{subjclassname@2020}{\textup{2020} Mathematics Subject Classification}
\makeatother
\begin{document}
\title[] {Average Nikolskii factors for random diffusion polynomials on closed Riemannian manifolds}

\author[]{Yun Ling} \address{School of Mathematical Sciences, Capital Normal
University, Beijing 10004,
China}
\email{lingyun0503@yeah.net}

\author[]{Heping Wang} \address{School of Mathematical Sciences, Capital Normal
University, Beijing 100048,
China}
\email{wanghp@cnu.edu.cn}

\begin{abstract}
For $1\le p,q\le \infty$, the Nikolskii factor for a diffusion
polynomial $P_{\a}$ of degree at most $n$ is defined by
$$N_{p,q}(P_{\a})=\frac{\|P_{\a}\|_{q}}{\|P_{\a}\|_{p}},\ \ P_{\a}(\x)=\sum_{k:\lambda_{k}\leq n}a_{k}\phi_{k}(\x),$$
where $\a=\{a_k\}_{\lz_k\le n}$,  and
$\{(\phi_k,-\lambda_k^2)\}_{k=0}^\infty$ are the eigenpairs of the
Laplace-Beltrami operator $ \Delta_{\mathbb M}$ on a closed smooth
Riemannian manifold $\mathbb M$ with normalized Riemannian measure.
 We study this average
Nikolskii factor for random diffusion polynomials with independent
$N(0,\sigma^{2})$ coefficients and obtain  the exact orders. For
$1\leq p<q<\infty$, the average Nikolskii factor is of order
$n^{0}$ (i.e., constant), as compared to the worst case bound of
order $n^{d(1/p-1/q)}$,  and for $1\leq p<q=\infty$,  the average
Nikolskii factor is of order $(\ln n)^{1/2}$ as compared to the
worst case bound of order $n^{d/p}$.
\end{abstract}
\maketitle
\input amssym.def

\section{Introduction}
Let $\Bbb M$ be a closed connected smooth $d$-dimensional
Riemannian manifold equipped with the normalized Riemannian
measure $\mu$. Here  a closed manifold is  a compact manifold
without boundary. For $1\le p<\infty$, let $L_p(\Bbb M)$ be the
set of all complex valued measurable function defined on $\Bbb M$
endowed with finite norm
$$\|f\|_{p}:=\left(\int_{\Bbb M}|f({\bf x})|^p{\rm d}\mu({\bf x})\right)^{1/p},$$
and let  $L_\infty(\Bbb M)\equiv C(\Bbb M)$ be the set of all complex valued continuous functions defined on $\Bbb M$ endowed with uniform norm
$$\|f\|_{\infty}:=\sup\limits_{{\bf x}\in\Bbb M}|f({\bf x})|.$$
 In particular, $L_2(\Bbb M)$ is a Hilbert space with inner product

\begin{equation}\label{1.1}
\langle
f,g\rangle:=\int_{\Bbb M}f({\bf x})\overline{g({\bf x})}{\rm d}\mu({\bf x}),\ {\rm for}\ f,g\in
L_2(\Bbb M),
\end{equation}
where $\overline{g}$ is the complex conjugate to $g$.

Let $\Delta_{\Bbb M}$ be the Laplace-Beltrami operator on $\Bbb M$, which is self-adjoint in $L_2(\Bbb M)$. This operator has a sequence of
eigenvalues $\{-\lambda_k^2\}_{k=0}^\infty$ arranged by
$$0=\lambda_0\le\lambda_1\le\dots\le\lambda_k\le\dots\to\infty,$$
and a complete orthonormal system of eigenfunctions
$\{\phi_k\}_{k=0}^{\infty}$. Without loss of generality, we choose
each $\phi_k$ to be real-valued and $\phi_0=1$. Then
$\{\phi_{k}\}^{\infty}_{k=0}$ is an orthonormal basis for
$L_{2}(\mathbb{M})$. For each $n=0,1,2,\dots$, with
$\mathcal{P}_n$ we denote the set of ``polynomials" of degree at
most $ n$ on $\Bbb M$ by
$$\mathcal{P}_n:={\rm span}\left\{\phi_k:\,\lambda_{k}\leq n\right\}.$$
This space is called the \emph{diffusion polynomial space of
degree $n$} on $\Bbb M$, and an element of $\mathcal{P}_n$ is
called a \emph{diffusion polynomial of degree at most $n$}. We
have  \emph{two-sides Wely's law} with exponent $d$ and constants
$c_{2}(d)>c_{1}(d)>0$ (see \cite{GL,HL,HO,SM}), i.e.,
\begin{equation}\label{1.2}
c_{1}(d)n^d\le\sum_{k:\lambda_{k}\leq n}|\phi_k({\bf x})|^2\le
c_{2}(d)n^d,\ \ {\bf x}\in \Bbb M,
\end{equation}
which, by integrating on $\Bbb M$ with respect to ${\bf x}$,
implies that the dimension $N$ of the space $\mathcal{P}_{n}$
satisfies
\begin{equation}\label{1.3}
N:=\dim \mathcal{P}_{n} := \#\{\phi_{k}:\lambda_{k}\leq n\}\asymp
n^d,
\end{equation}
where $\#E$ denotes the number of element in the set $E$. Here we
use the notation $A_n\asymp B_n$ to express $A_n\ll B_n$ and
$A_n\gg B_n$, and $A_n\ll B_n$ ($A_n\gg B_n$) means that there
exists a constant $c>0$ independent of $n$ such that $A_n\leq
cB_n$ ($A_n\geq cB_n$).

Let $P$ be a {diffusion polynomial of degree at most $n$}.
Then,  $P$ is of the form $P_{\a}$
\begin{equation}\label{1.4}
P(\x)=P_{\a}(\x)=\sum_{k:\lambda_{k}\leq n}a_k\phi_k(\x), \ \
\x\in\mathbb{M},
\end{equation}
where $\a=\{a_k\}_{\lambda_{k}\leq n}\in\Bbb R^N$.

The classical Nikolskii inequality for diffusion polynomials
(see \cite[Lemma 5.5]{M}) states that there exists a positive constant
$C(d)$ depending only on $d$ such that for any $1\leq
p,q\leq\infty$ and $P_{\a}\in\mathcal{P}_{n}$,
\begin{equation}\label{1.5}
\|P_{\a}\|_{q}\leq C(d)n^{d(1/p-1/q)_+}\|P_{\a}\|_p,
\end{equation}
where $a_+=\max(a,0)$ for a real number $a$.  The order $d(1/p-1/q)_+$ in the
above inequality \eqref{1.5} is sharp.

For given $1\le p,q\le \infty$,  we introduce the Nikolskii
factor for $P_\a$ by
$${N}_{p,q}(P_{\a}):=\frac{\|P_{\a}\|_{q}}{\|P_{\a}\|_{p}}.$$
The worst Nikolskii factor for $\mathcal{P}_n$ is defined by
$$N^{\rm wor}_{p,q}(\mathcal{P}_{n}):=\sup_{0\neq P_{\a}\in\mathcal{P}_{n}}{N}_{p,q}
(P_{\a})=\sup_{0\neq
P_{\a}\in\mathcal{P}_{n}}\frac{\|P_{\a}\|_{q}}{\|P_{\a}\|_{p}},$$
where the supremum is taken over all diffusion polynomials of
degree $\le n$ whose coefficients are not all zero. Then the above
Nikolskii inequality \eqref{1.5} states
\begin{equation}\label{1.6}
N^{\rm
wor}_{p,q}(\mathcal{P}_{n})\asymp n^{d(1/p-1/q)_{+}},
\end{equation}
i.e., that the worst case bound for
${N}_{p,q}(P_{\a})$ is $\Theta(n^{d(1/p-1/q)_{+}})$.

Now we introduce average Nikolskii factors on the   compact
manifold $\Bbb M$. Average Nikolskii factors  were first
introduced and studied in \cite{LGLW}.   Average Nikolskii factors
on the torus and the sphere were investigated in \cite{LGLW} and
\cite{LGLW1}, respectively.  A random diffusion polynomial
$P_{\a}\in\mathcal{P}_{n}$ is defined by
\begin{equation}\label{1.7}
P_{\a}(\x)=\sum\limits_{k:\lambda_{k}\leq n}a_{k}\phi_{k}(\x),
\end{equation}
where $\a=\{a_{k}\}_{\lambda_{k}\leq n}\in \mathbb{R}^{N}$, the
coefficients $a_k$ are independent $N(0,\sigma^2)$ random
variables with the common normal density functions
$$\frac1{\sqrt{2\pi}\sz}e^{-\frac{a_k^2}{2\sz^2}}.$$
That is,  $\a\sim
N(0, \sz^2I_N)$ is an $\mathbb{R}^{N}$-valued Gaussian random
vector with mean $0$ and covariance matrix $\sigma^{2}I_{N}$,
where $I_N$ is the $N$ by $N$ identity matrix. Let $\gamma_{N}$ be the distribution or the law of the random vector $\a$. Then $\gamma_{N}$ is just the centered Gaussian measure on $\mathbb{R}^{N}$,
given by
$$\gamma_{N}(G)=\frac{1}{(2\pi\sigma^{2})^{N/2}}
\int_{G}e^{-\frac{|\a|^{2}}{2\sigma^{2}}}\,d\a\ {\rm for\
each\ Borel\ subset}\  G\subset\mathbb{R}^{N},$$
where
$|\a|=(\sum_{k:\lambda_{k}\leq n}a^{2}_{k})^{1/2}$.

For given $1\le p,q\le \infty$, the average Nikolskii
factor for $\mathcal{P}_{n}$ is defined by
$$N^{{\rm ave}}_{p,q}(\mathcal{P}_{n}):=\Bbb
E\frac{\|P_{\a}\|_{q}}{\|P_{\a}\|_{p}}=\int_{\mathbb{R}^{N}}
\frac{\|P_{\a}\|_{q}}{\|P_{\a}\|_{p}}d\gamma_{N}(\a)=\frac{1}{(2\pi\sigma^{2})^{N/2}}
\int_{\Bbb R^N}
\frac{\|P_{\a}\|_{q}}{\|P_{\a}\|_{p}}e^{-\frac{|\a|^{2}}{2\sigma^{2}}}
d\a.$$

It follows immediately from the definitions of the two Nikolskii
factors that for $1\le p,q\le \infty$,
$$N^{{\rm ave}}_{p,q}(\mathcal{P}_{n})\le N^{{\rm
wor}}_{p,q}(\mathcal{P}_{n}).$$

This paper is devoted to investigating  expected values of the
Nikolskii factors of random diffusion polynomials. Given a
polynomial basis $p_1,\dots, p_N$, a random polynomial $p$ is a
polynomial whose coefficients are random, i.e.,
$$p=\sum_{k=1}^NX_kp_k,$$
where $X_k,\,1\le k\le N$ are independent and identically distributed (i.i.d.) random variables. In recent years the study of random polynomials
has attracted much interest, and there are numerous papers devoted
to this field, see \cite{BL, Bos,  BLR,  CE, CP, DM, DNN, DNV, F,
G, LWW, LGLW, LGLW1, NB, Pr, PR, SZ, SW, WYZ}.

There are also many papers devoted to studying the Nikolskii type
inequalities, see for example, \cite{AD, D, DT, GT, T}. However,
as far as we know, there are only two papers devoted to
investigating the average Nikolskii factors for random polynomials
(see \cite{LGLW,LGLW1}). Indeed, the authors in \cite{LGLW}
studied the the average Nikolskii factors on the $d$-dimensional
torus and obtained their exact orders, and later the authors in
\cite{LGLW1} studied the the average Nikolskii factors on the
$(d-1)$-dimensional sphere and obtained their exact orders. It
turns out that  in many cases, the average Nikolskii factors on
the $d$-dimensional torus and the $(d-1)$-dimensional sphere are
significantly smaller than the worst Nikolskii factors.

The purpose of this paper is to extend these results to the closed
Riemannian manifold. We  investigate the average Nikolskii factors
$N^{{\rm ave}}_{p,q}(\mathcal{P}_{n})$ on the closed manifold
$\mathbb{M}$, and show that in many cases, the average Nikolskii
factor $N^{\rm ave}_{p,q}(\mathcal{P}_{n})$  is significantly
smaller than the worst Nikolskii factor $N^{\rm
wor}_{p,q}(\mathcal{P}_{n})$.

\begin{thm}\label{thm1.1}
Let  $1\leq p,q\leq\infty$. Then we have
\begin{equation}\label{1.8}
N^{\rm ave}_{p,q}(\mathcal{P}_{n})=\mathbb{E}\frac{\|P_{\a}\|_{q}}{\|P_{\a}\|_{p}}\asymp\left\{
\begin{aligned}
 &1,&&1\leq p,q<\infty\ or\ p=q=\infty,\\
 &(\ln n)^{1/2},&&1\leq p<q=\infty,\\
 &(\ln n)^{-1/2},&&1\leq q<p=\infty.
\end{aligned}
\right.
\end{equation}
\end{thm}

\begin{rem}
Comparing \eqref{1.6} and \eqref{1.8}, we see that the orders of
the worst and average  Nikolskii factors are the same for $1\leq
q\leq p<\infty$ or $p=q=\infty$, whereas   the order of average
Nikolskii factors are significantly smaller than the worst
Nikolskii factors for the other cases.  More precisely, for $1\leq
p<q<\infty$, the worst Nikolskii factor is of order
$n^{d(1/p-1/q)}$, whereas the average Nikolskii factor is of order
$n^{0}$, for $1\leq p<q=\infty$, the worst Nikolskii factor is of
order $n^{d/p}$, whereas the average Nikolskii factor is of order
$(\ln n)^{1/2}$, and for $1\le q<p=\infty$, the worst Nikolskii
factor is of order $n^{0}$, whereas the average Nikolskii factor
is of order $(\ln n)^{-1/2}$.

This indicates that for $1\le p<q\le \infty$ or $1\le q<p=\infty$,
the order of the average Nikolskii factor
 is significantly smaller than the worst case Nikolskii factor.
\end{rem}

The remaining paper is structured as follows. Section
\ref{section2} contains some preliminary materials about harmonics
analysis on closed manifolds and some auxiliary results. In
Section \ref{section3}, we determine the exact order of the
average Nikolskii factor in the case of $1\leq q\leq\infty,p=2$
and give the upper bound of
$\mathbb{E}(\|P_{\a}\|^{-r}_{\infty})$. Finally, we prove Theorem
\ref{thm1.1} in Section \ref{section4}.

\section{Preliminaries}\label{section2}
In this section, we will provide  necessary materials required for
the remainder of the paper. We will begin by providing an overview
of fundamental concepts pertaining to closed Riemannian manifolds.
Subsequently, we will introduce the concepts of reproducing kernel
and Christoffel function. We will also present necessary knowledge
and establish key lemmas.

Let $\Bbb M$ be a closed connected smooth $d$-dimensional
Riemannian manifold equipped with the normalized Riemannian
measure $\mu$. Typical examples are closed  homogeneous manifolds,
which is of the form $G/H$. Here $G$ is a compact Lie group and
$H$ is a closed subgroup of $G$. For example, the $d$-torus $\Bbb
T^d$, $d=1,2,3,\dots$; the spheres $\Bbb S^d$, $d=1,2,3,\dots$; the
real projective spaces ${\rm P}^d(\Bbb R),\ d=2,3,4,\dots$; the
complex projective spaces ${\rm P}^{d}(\Bbb C),\ d=4,6,8,\dots$;
the quaternion projective spaces ${\rm P}^{2d}(\Bbb H),\
d=4,6,8,\dots$; the Cayley elliptic plane ${\rm P}^{16}(Cay)$. See
\cite{C,Ga,H,H1,H2} for details.

Let ${\rm d}({\bf x},{\bf y})$ be the geodesic distance of points
${\bf x}$ and ${\bf y}$ on $\Bbb M$, and let ${\rm B}({\bf x},r)$
denote the ball of radius $r>0$ centred at ${\bf x}\in\Bbb M$.
Then   $(\Bbb M,\mu)$ is  \emph{Ahlfors $d$-regular} which means
there exists a constant $C_{d}>0$ depending only on the dimension
$d$ such that
\begin{equation}\label{2.1}
C_{d}^{-1}\,r^d\le \mu\left({\rm B}({\bf x},r)\right)\le
C_{d}\,r^d,\ {\rm for\ all}\ {\bf x}\in\Bbb M\ {\rm and}\
0<r\le{\rm diam}(\Bbb M),
\end{equation}
where $\mu(E)=\int_E{\rm d}\mu({\bf x})$ for a measurable subset
$E$ of $\Bbb M$, and ${\rm diam}(\Bbb M):=\sup\limits_{
\x,\y\in\mathbb{M}}{\rm d}(\x,\y)$.

Let $\{(\phi_k,-\lambda_k^2)\}_{k=0}^\infty$ be the eigenpairs of
the Laplace-Beltrami operator $ \Delta_{\Bbb M}$   on $\Bbb M$.
Then $\{\phi_k\}_{k=0}^{\infty}$ is  an orthonormal basis  for
$L_{2}(\mathbb{M})$, and
 $\{\phi_k\}_{\lambda_{k}\leq n}$ is  an orthonormal basis for  the
space $\mathcal{P}_{n}$ of diffusion polynomials of degree at most
$n$ on $\mathbb{M}$. We define the Schwartz kernel of
$\mathcal{P}_{n}$  by
$$e(\x,\y,n)=\sum_{k:\lambda_{k}\leq n}\phi_{k}(\x)\phi_{k}(\y),\ \
\x,\y\in \Bbb M.$$
In particular, when  $\x=\y$,  by \eqref{1.2} we obtain
$$e(\x,\x,n)=\sum_{k:\lambda_{k}\leq n}|\phi_k({\bf x})|^2\asymp n^d,\ \ \x\in \Bbb M.$$
By the Cauchy-Schwartz inequality we have
\begin{equation}\label{2.2}
|e(\x,\y,n)|\leq \sqrt{e(\x,\x,n)}\sqrt{e(\y,\y,n)}\ll n^d.
\end{equation}

Let  $$P_{\a}(\x)=\sum\limits_{k:\lambda_{k}\leq
n}a_{k}\phi_{k}(\x)$$
be the random diffusion polynomial on $\Bbb M$, where $\a=\{a_{k}\}_{\lambda_{k}\leq n}\sim
N(0, \sz^2I_N)$ is a $\mathbb{R}^{N}$-valued Gaussian random
vector with mean $0$ and covariance matrix $\sigma^{2}I_{N}$. Then
the random diffusion polynomial $P_{\a}$
is centered, and its distribution is determined by its covariance
function $K(\x,\y)$, which is precisely $\sigma^2e(\x,\y,n)$.
Indeed, we have
\begin{align*}
K(\x,\y)&:=\Bbb E(P_{\a}(\x)P_{\a}(\y))=\sum\limits_{k:\lambda_{k}\leq
n}\sum\limits_{j:\lambda_{j}\leq n}\Bbb E(a_{k} a_{j})
\phi_{k}(\x) \phi_{j}(\y)\\
&=\sum\limits_{k:\lambda_{k}\leq
n}\sigma^2 \phi_{k}(\x)
\phi_{k}(\y)=\sigma^2e(\x,\y,n).
\end{align*}

We define the Bessel function $J_{v}$ of order $v$ by its Poisson
representation formula
$$J_{v}(t)=\frac{(\frac{t}{2})^{v}}{\Gamma(v+\frac{1}{2})\Gamma(\frac{1}{2})}\int^{1}_{-1}\cos(ts)(1-s^{2})^{v-1/2}ds,$$
where $v>-1/2$ and $t\geq 0$. We have the following well-known
asymptotic formula (see \cite[Appendix B]{GR})
\begin{equation}\label{2.3}
J_{v}(t)=\sqrt{\frac{2}{\pi t}}\cos\Big(t-\frac{\pi v}{2}-\frac{\pi}{4}\Big)+O(t^{-3/2}),\quad t\rightarrow+\infty.
\end{equation}
Consider the radial function $F_{d}({\bf z})$ on $\mathbb{R}^{d}$
defined by
$$F_{d}({\bf z}):=\Phi_{d}(|{\bf z}|):=(2\pi)^{-d}\int_{\mathbb{B}^{d}}e^{i\bf z\cdot\y}d\y,$$
where $\mathbb{B}^{d}:=\{\x\in\mathbb{R}^{d}:|\x|\leq1\}$ and ${\bf z}\cdot\y=\sum^{d}_{j=1}z_{j}y_{j}$ is the inner
product in $\mathbb{R}^{d}$.
 $\Phi_{d}(t)$, $t>0$, can be rewritten as
\begin{equation}\label{2.4}
\Phi_{d}(t)=(2\pi)^{-d}\omega_{d-1}\int^{1}_{-1}\cos(ts)(1-s^{2})^{(d-1)/2}ds=\frac{J_{d/2}(t)}{(2\pi)^{d/2}t^{d/2}},
\end{equation}
where $\omega_{d-1}=\frac{2\pi^{d/2}}{\Gamma(d/2)}$ is the surface
area of  the unit sphere $\mathbb{S}^{d-1}$ in $\mathbb{R}^{d}$.
Furthermore, we have the following more precise estimate for the
Schwartz kernel $e(\x,\y,n)$.
\begin{lem}\label{lem2.1} {\rm (\cite[Theorem 3.3]{X})}
As $n\rightarrow\infty$, there holds the asymptotic formula
\begin{equation}\label{2.5}
e(\x,\y,n)=\left\{
\begin{aligned}
 &\Phi_{d}(n{\rm d}(\x,\y))n^{d}+O(n^{d-1}),&&\Phi_{d}(nd(\x,\y))\neq0,\\
 &O(n^{d-1}),&& {\rm otherwise}.
\end{aligned}
\right.
\end{equation}
\end{lem}

For given $1\le p,q\le
\infty$, the average Nikolskii factor for $\mathcal{P}_n$
is defined by
\begin{align*}
N^{{\rm ave}}_{p,q}(\mathcal{P}_n):=\Bbb
E\frac{\|P_{\a}\|_{q}}{\|P_{\a}\|_{p}}=\frac{1}{(2\pi\sigma^{2})^{N/2}}
\int_{\Bbb R^N}
\frac{\|P_{\a}\|_{q}}{\|P_{\a}\|_{p}}e^{-\frac{|\a|^{2}}{2\sigma^{2}}}
d\a.
\end{align*}
It follows from \cite[Proposition 2.1]{LGLW} that $$N^{{\rm
ave}}_{p,q}(\mathcal{P}_n)=\Bbb
E\frac{\|P_{\b}\|_{q}}{\|P_{\b}\|_{p}},$$
where
$\b=\{b_{k}\}_{\lambda_{k}\leq n}\sim N(0, I_N)$. Hence, in the
sequel  we always assume that $\sz=1$ and $\a\sim N(0,
 I_N)$ without loss of generality.

The following lemmas will be used in the following sections.
\begin{lem} {\rm (\cite[Theorem 2.2]{LGLW})} \label{lem2.2}
Let $1\leq p,q\leq\infty$ and $k,l>0$. Then for $l<k+N$,
\begin{equation}\label{2.6}
\mathbb{E}\frac{\|P_{\a}\|^{k}_{q}}{\|P_{\a}\|^{l}_{2}}\asymp
n^{-dl/2}\mathbb{E}\|P_{\a}\|^{k}_{q},
\end{equation}
and for $l<N$,
\begin{equation}\label{2.7}
\mathbb{E}\frac{\|P_{\a}\|^{k}_{2}}{\|P_{\a}\|^{l}_{p}}\asymp
n^{dk/2}\mathbb{E}\|P_{\a}\|^{-l}_{p}.
\end{equation}
In particular, for $s>0$,
\begin{equation}\label{2.8}
\mathbb{E}\frac{\|P_{\a}\|^s_{q}}{\|P_{\a}\|^s_{2}}\asymp
n^{-ds/2}\mathbb{E}\|P_{\a}\|_{q}^s.
\end{equation}
\end{lem}

Next, we estimate $\mathbb{E}\|P_{\a}\|_{q}^s$, $1\leq s,
q<\infty$. This quantity is intimately related to the Christoffel
function for $\mathcal{P}_n$. We recall that the Christoffel function
for $\mathcal{P}_n$ is defined by
$$\Lz(\x):=\min_{p(\x)=1,\ p\in\mathcal{P}_n}\|p\|_2^2,\quad\x\in\mathbb{M},$$
where the minimum is taken over all $p\in\mathcal{P}_n$ that take
the value $1$ at $\x$. It follows from \cite[Theorem 3.5.6]{DX1}
and \eqref{1.2} that
\begin{equation}\label{2.9}
\Lz(\x)=\Big(\sum_{k:\lambda_{k}\leq n}|\phi_k(\x)|^{2}\Big)^{-1}\asymp n^{-d},\quad\x\in\mathbb{M},
\end{equation}
which  means that the Christoffel function for $\mathcal{P}_n$ is
asymptotically proportional to the reciprocal of the $n$ raised to
the power of the dimension $d$.

According to \cite[Corollary 2.4]{LGLW} and \eqref{2.9}, we obtain the
following lemma.
\begin{lem}\label{lem2.3}
Let $1\leq s, q<\infty$. Then
\begin{equation}\label{2.10}
\mathbb{E}\|P_{\a}\|^{s}_{q}\asymp
\Big\|\Big(\sum\limits_{k:\lambda_{k}\leq n}
|\phi_{k}|^{2}\Big)^{1/2}\Big\|^{s}_{q}=\|\Lambda^{-1/2}\|^{s}_{q}\asymp n^{ds/2}.
\end{equation}
\end{lem}

\section{Average Nikolskii factors for $\mathcal{P}_{n}$}\label{section3}

This section is devoted to giving asymptotic estimates of
$N^{\rm{ave}}_{2,q}(\mathcal{P}_{n})$ $(1\leq q\leq\infty)$ and
upper bound estimates of
$N^{\rm{ave}}_{\infty,2}(\mathcal{P}_{n})$ by utilising the upper
bound estimate of Schwartz kernel introduced in Section
\ref{section2}, the Marcinkiewicz-Zygmund inequalities on
$\mathbb{M}$, and the probability estimate techniques.

Before we proceed to state our main result, we introduce the concept of a maximal separated subset on $\mathbb{M}$. Let $\Xi$ be a finite subset of $ \mathbb{M}$. We say that $\Xi$
is \emph{$\varepsilon$-separated} if the minimum distance between
any two distinct points $\boldsymbol{\xi}$ and $\boldsymbol{\eta}$
in $\Xi$ is at least $\varepsilon$. In other words, we require
$$\min\limits_{\substack{\boldsymbol{\xi},\boldsymbol{\eta}\in\Xi
\\ \boldsymbol{\xi}\neq\boldsymbol{\eta}}}{\rm d}(\boldsymbol{\xi},\boldsymbol{\eta})\geq\varepsilon.$$
Furthermore, if an $\varepsilon$-separated subset $\Xi$ of
$\mathbb{M}$ satisfies
$$\max\limits_{\boldsymbol{\xi}\in\mathbb{M}}\inf_{\boldsymbol\eta\in \Xi}{\rm d}(\boldsymbol{\xi},\boldsymbol\eta)\leq\varepsilon,$$
then we refer to it as a \emph{maximal $\varepsilon$-separated} subset of
$\mathbb{M}$. In simpler terms, a maximal $\varepsilon$-separated subset $\Xi$ is
an $\varepsilon$-separated subset of $\mathbb{M}$ such that every point in $\mathbb{M}$
lies within a distance of $\varepsilon$ from at least one point in $\Xi$.
We can express this as
$$\mathbb{M}=\bigcup\limits_{\boldsymbol{\xi}\in\Xi}{\rm B}(\boldsymbol{\xi},\varepsilon).$$

We conclude from the Ahlfors $d$-regularity \eqref{2.1} that if
$\Xi$ is maximal $\vz$-separated then
\begin{equation}\label{3.1}
M:=\# \Xi \asymp \varepsilon^{-d},
\end{equation}
which  means that the cardinality of $\Xi$ is asymptotically
proportional to the reciprocal of the $\varepsilon$ raised to the
power of the dimension $d$.

Next, we present our main result in this section.
\begin{thm}\label{thm3.1}
Let $p=2$ and $1\leq q\leq\infty$. Then
\begin{equation}\label{3.2}
N^{\rm{ave}}_{2,q}(\mathcal{P}_{n})=
\mathbb{E}\frac{\|P_{\a}\|_{q}}{\|P_{\a}\|_{2}}\asymp
n^{-d/2}\Bbb E\|P_{\a}\|_q\asymp\left\{
\begin{aligned}
&1,&&1\leq q<\infty,\\
&\sqrt{\ln n},&&q=\infty,
\end{aligned}
\right.
\end{equation}
where the equivalent constants are independent of $n$.
\end{thm}
\begin{proof}
For $1\leq q<\infty$, \eqref{2.8} and \eqref{2.10} give
$$N^{\rm{ave}}_{2,q}(\mathcal{P}_{n})\asymp n^{-d/2}\Bbb
E\|P_{\a}\|_q\asymp 1.$$

It remains to estimate $N^{\rm{ave}}_{2,\infty}(\mathcal{P}_{n})$.
We need the following two lemmas. The first lemma gives the
estimate of the expectation of the uniform norm of a centered
Gaussian random vector.
\begin{lem}\label{lem3.2} {\rm (\cite[(5.23)]{P})}
Let $X=\{X_{j}\}^{M}_{j=1}$ is a centered  Gaussian random vector.
Then
$$C_{1}(\ln M)^{1/2}\inf_{1\leq i\neq j\leq
M}\|X_{i}-X_{j}\|_2\leq\mathbb{E}\max\limits_{1\leq i\leq
M}|X_{i}|\leq C_{2}(\ln M)^{1/2}\sup_{1\leq i\neq j\leq
M}\|X_{i}-X_{j}\|_2$$
for some absolute constants $C_{1}$ and
$C_{2}$, where $\|X_i-X_j\|_2=\big(\Bbb E |X_i-X_j|^2\big)^{1/2}$.
\end{lem}

The second lemma is about  Marcinkiewicz-Zygmund inequalities on
$\mathbb{M}$.
\begin{lem}{\rm (\cite[Theorem 5.1]{FM}).}\label{lem3.3}
Let $1\le p\le\infty$. Then there exists a constant $\delta_0>0$ depending only on $d$ and $p$ such that
given a maximal $\delta/n$-separated subset $\Xi$ of $\Bbb M$ with $\delta\in(0,\delta_0]$,
there exists a set of positive numbers $\{\lambda_{\boldsymbol\xi}\}_{ \boldsymbol\xi\in\Xi}$ such that for all  $f\in\mathcal{P}_n$,
$$\|f\|_{p}\asymp \Bigg\{\begin{aligned}&\Big(\sum_{\boldsymbol\xi\in\Xi}\lambda_{\boldsymbol\xi}|f(\boldsymbol\xi)|^p\Big)^{1/p},  &&1\le p<\infty,\\
&\max\limits_{\boldsymbol\xi\in\Xi}|f(\boldsymbol\xi)|, &&p=\infty,
\end{aligned}$$
and $\lambda_{\boldsymbol\xi}\asymp \mu \left({\rm B}(\boldsymbol\xi,\delta/n)\right)\asymp n^{-d}$
for all $\boldsymbol\xi\in\Xi$, where the constants of equivalence depend only on $d$ and $p$.
\end{lem}
Let $\Xi=\{{\boldsymbol\xi_{1}},\ldots,{\boldsymbol\xi_{M}}\}$ be
a maximal $\frac{\delta_0}{n}$-separated subset of $\mathbb{M}$,
i.e.,
$${\rm d}({\boldsymbol\xi_{i}},{\boldsymbol\xi_{j}})\geq\frac{\delta_0}{n}\ \ (i\neq j)$$ and
$$\bigcup^{M}_{j=1}{\bf
 B}({\boldsymbol\xi_{j}},\frac{\delta_0}{n})=\mathbb{M}.$$

We note that $M\asymp N\asymp n^{d}$, where $M=\#\Xi$ is the
cardinality of the set $\Xi$. We set
$$X=\{X_{j}\}^{M}_{j=1},\ \ X_j=\frac{P_{\a}({\boldsymbol\xi_j})}{\sqrt{e({\boldsymbol\xi_{j}},{\boldsymbol\xi_{j}},n)}},\ \ 1\leq j\leq M,$$
where $P_{\a}$ is the random diffusion polynomial given by
\eqref{1.7} and $\a\sim N(0,I_N)$. Then, $X$ is a  centered
Gaussian random vector with covariance matrix
$C=(C_{i,j})_{i,j=1}^M$, where $C_{i,j}=\Bbb E(X_i X_j),\ 1\le
i,j\le M$. We note that
$$\Bbb E(a_{k}
 a_{l})=\left\{
\begin{aligned}
 &0,&\ \,k\neq l,\\
 &1,&\ \,k=l.
\end{aligned}
\right.$$
It follows that for $1\leq j\leq M$,
$$\mathbb{E}X^{2}_{j}=\frac{\mathbb{E}P^{2}_{\a}({\boldsymbol\xi_j})}{e({\boldsymbol\xi_{j}},{\boldsymbol\xi_{j}},n)}=\frac{\sum_{k:\lambda_{k}\leq n}|\phi_{k}({\boldsymbol\xi_{j}})|^{2}}{e({\boldsymbol\xi_{j}},{\boldsymbol\xi_{j}},n)}=1,$$
giving $X_{j}\sim N(0,1)$. This implies that
\begin{align}\label{3.3}
\Bbb{E}(X_{i}-X_{j})^{2}&=\Bbb{E}\bigg(\frac{P_{\a}({\boldsymbol\xi_{i}})}{\sqrt{e({\boldsymbol\xi_{i}},{\boldsymbol\xi_{i}},n)}}-\frac{P_{\a}({\boldsymbol\xi_{j}})}{\sqrt{e({\boldsymbol\xi_{j}},{\boldsymbol\xi_{j}},n)}}\bigg)^{2}\notag\\
&=\mathbb{E}\bigg(\sum_{k:\lambda_{k}\leq n}a_k\bigg(\frac{\phi_{k}({\boldsymbol\xi_{i}})}{\sqrt{e({\boldsymbol\xi_{i}},{\boldsymbol\xi_{i}},n)}}-\frac{\phi_{k}({\boldsymbol\xi_{j}})}{\sqrt{e({\boldsymbol\xi_{j}},{\boldsymbol\xi_{j}},n)}}\bigg)\bigg)^{2}\notag\\
&=\sum_{k:\lambda_{k}\leq n}\bigg(\frac{\phi_{k}({\boldsymbol\xi_{i}})}{\sqrt{e({\boldsymbol\xi_{i}},{\boldsymbol\xi_{i}},n)}}-\frac{\phi_{k}({\boldsymbol\xi_{j}})}{\sqrt{e({\boldsymbol\xi_{j}},{\boldsymbol\xi_{j}},n)}}\bigg)^{2}\notag\\
&=2-\frac{2e({\boldsymbol\xi_{i}},{\boldsymbol\xi_{j}},n)}{\sqrt{e({\boldsymbol\xi_{i}},{\boldsymbol\xi_{i}},n)e({\boldsymbol\xi_{j}},{\boldsymbol\xi_{j}},n)}}.
\end{align}
By \eqref{3.3} and \eqref{2.2} we obtain
$$\sup_{i\neq j}
\|X_{i}-X_{j}\|_2\leq2.$$
Thus, using \eqref{2.8}, Lemmas
\ref{lem3.3} and \ref{lem3.2}, we obtain
\begin{align}\label{3.4}
N^{\rm{ave}}_{2,\infty}(\Pi^{d}_{n})&\asymp n^{-d/2}\Bbb
E\|P_{\a}\|_\infty\asymp  n^{-d/2}\Bbb E \max_{1\le j\le
M}|P_{\a}({\boldsymbol\xi_{j}})|\asymp\Bbb E\max_{1\le j\le M}|X_j|\notag\\
&\ll\sqrt{\ln M}\sup_{i\neq j} \|X_{i}-X_{j}\|_2\ll\sqrt{\ln n}.
\end{align}

Now we want to find a subset $\widetilde\Xi=\{\tilde
{\boldsymbol\xi}_1,\dots,\tilde{\boldsymbol\xi}_{\widetilde M}\}$
of $\Xi$ such that $\widetilde \Xi$ is $\frac\delta n$-separated
and
$$\bigcup_{\tilde{\boldsymbol\xi}\in\widetilde\Xi}B(\tilde{\boldsymbol\xi},\frac{2\delta}n)=\Bbb
M,$$
 where $\delta>\delta_0$ is a sufficiently large constant chosen later. Such $\widetilde\Xi$
 exists and $\widetilde M=\#\widetilde \Xi\asymp n^d$.
Indeed, choose arbitrary  $\tilde{\boldsymbol\xi}_1$ in $\Xi$. Set
$$A_{\tilde{\boldsymbol\xi}_1}:= \{\tilde{\boldsymbol\xi}\in \Xi \
:\ d(\tilde{\boldsymbol\xi}_1, \tilde{\boldsymbol\xi})\ge
\frac{\delta}n\}.$$
If $A_{\tilde{\boldsymbol\xi}_1}\neq
\emptyset$, then there exists ${\tilde{\boldsymbol\xi}_2\in
A_{\tilde{\boldsymbol\xi}_1}}$. Set
$$A_{\tilde{\boldsymbol\xi}_2}:= \{\tilde{\boldsymbol\xi}\in A_{\tilde{\boldsymbol\xi}_1} \
:\ d(\tilde{\boldsymbol\xi}_2, \tilde{\boldsymbol\xi})\ge
\frac{\delta}n\}.$$ Continuing the above step until
$$A_{\tilde{\boldsymbol\xi}_{\widetilde M-1}}\neq \emptyset\ \
{\rm and}\ \ A_{\tilde{\boldsymbol\xi}_{\widetilde M}}=
\emptyset.$$
Then the set  $\widetilde\Xi=\{\tilde
{\boldsymbol\xi}_1,\dots,\tilde{\boldsymbol\xi}_{\widetilde M}\}$
is a subset of $\Xi$ and $\frac{\delta}n$-separated. It follows
that $\widetilde M\ll n^d$. Since
$A_{\tilde{\boldsymbol\xi}_{\widetilde M}}= \emptyset,$ we get
$$\Xi\subset \bigcup_{\tilde{\boldsymbol\xi}\in
\widetilde\Xi}B(\tilde{\boldsymbol\xi},\frac{\delta} n).$$
We recall that $ \bigcup_{{\boldsymbol\xi}\in
\Xi}B({\boldsymbol\xi},\frac{\delta_0} n)=\Bbb M$. For $\x\in \Bbb
M$, there exists ${\boldsymbol\xi}\in \Xi$ such that
${\rm d}(\x,{\boldsymbol\xi})\le \frac{\delta_0}n$. For such
${\boldsymbol\xi}$, there exists $\tilde {\boldsymbol\xi}\in
\widetilde\Xi$ such that
${\rm d}({\boldsymbol\xi},\tilde{\boldsymbol\xi})\le \frac\delta n$. By
the triangular inequality we get ${\rm d}(\x,\tilde{\boldsymbol\xi})\le
\frac{\delta+\delta_0}n\le \frac{2\delta}n$. It follows that
$$\bigcup_{\tilde{\boldsymbol\xi}\in\widetilde\Xi}B(\tilde{\boldsymbol\xi},\frac{2\delta}n)=\Bbb
M.$$
Hence, we have
$$1\asymp
\mu(M)=\mu\bigg(\bigcup_{\tilde{\boldsymbol\xi}\in\widetilde\Xi}B(\tilde{\boldsymbol\xi},\frac{2\delta}n)\bigg)\le
\sum_{\tilde{\boldsymbol\xi}\in\widetilde\Xi}\mu\Big(B(\tilde{\boldsymbol\xi},\frac{2\delta}n)\Big)\ll
\# \widetilde\Xi\cdot \Big( \frac{2\delta}n\Big)^d,$$
giving
$\widetilde M=\# \widetilde\Xi\gg n^d.$

We set
$$\widetilde X=\{\widetilde X_{j}\}^{\widetilde M}_{j=1},\ \ \widetilde X_j=
\frac{P_{\a}(\tilde{\boldsymbol\xi}_j)}{\sqrt{e(\tilde{\boldsymbol\xi}_{j},\tilde{\boldsymbol\xi}_{j},n)}},\
\ 1\leq j\leq \widetilde M,$$
where $P_{\a}$ is the random
diffusion polynomial given by \eqref{1.7} and $\a\sim N(0,I_N)$.
Then, $\widetilde X$ is a  centered Gaussian random vector with
covariance matrix $\widetilde C=(\widetilde
C_{i,j})_{i,j=1}^{\widetilde M}$, where $\widetilde C_{i,j}=\Bbb
E(\widetilde X_i \widetilde X_j),\ 1\le i,j\le \widetilde M$. We
also have $\widetilde X_j\sim N(0,1)$ and for $1\le i\neq j\le
\widetilde M$
\begin{equation}\label{3.5}
\Bbb{E}(\widetilde X_{i}-\widetilde
X_{j})^{2}=2-\frac{2e(\tilde{\boldsymbol\xi}_{i},\tilde{\boldsymbol\xi}_{j},n)}{\sqrt{e(\tilde{\boldsymbol\xi}_{i},
\tilde{\boldsymbol\xi}_{i},n)e(\tilde{\boldsymbol\xi}_{j},\tilde{\boldsymbol\xi}_{j},n)}}.
\end{equation}

 Using \eqref{2.3}, there exists a positive number $t_{0}>\delta_0$ such that $t>t_{0}$,
$$J_{d/2}(t)\ll t^{-1/2},$$
and thus, together with \eqref{2.4}, when $n{\rm
d}({\boldsymbol\xi},{\boldsymbol\eta})>t_{0},\
{\boldsymbol\xi},{\boldsymbol\eta}\in \Bbb M$,
\begin{equation}\label{3.6}
\Phi_{d}(n{\rm d}({\boldsymbol\xi},{\boldsymbol\eta}))\leq
c(d)(n{\rm
d}({\boldsymbol\xi},{\boldsymbol\eta}))^{-\frac{d+1}{2}}.
\end{equation}
Choose $\delta>t_{0}$ big enough such that
$$\frac{c_{1}(d)}{4}+c(d)\delta^{-\frac{d+1}{2}}\leq\frac{c_{1}(d)}{2},$$
where $c_{1}(d)$ is given by \eqref{1.2}. Hence, by \eqref{2.5}
and \eqref{3.6} we have for $1\le i\neq j\le \widetilde M$ and
sufficiently large $n$,
\begin{align*}
e(\tilde {\boldsymbol\xi}_{i},\tilde{\boldsymbol\xi}_{j},n)&\leq\Big(\frac{c_{1}(d)}{4}+\Phi_{d}(n{\rm d}(\tilde{\boldsymbol\xi}_{i},
\tilde{\boldsymbol\xi}_{j}))\Big)n^{d}\leq\Big(\frac{c_{1}(d)}{4}+c(d)(n{\rm d}(\tilde{\boldsymbol\xi}_{i},\tilde{\boldsymbol\xi}_{j}))^{-\frac{d+1}{2}}\Big)n^{d}\\
&\leq\Big(\frac{c_{1}(d)}{4}+c(d)\delta^{-\frac{d+1}{2}}\Big)n^{d}\leq\frac{c_{1}(d)}{2}n^{d},
\end{align*}
where in the third inequality, we used the fact that $n{\rm
d}(\tilde{\boldsymbol\xi_{i}},\tilde{\boldsymbol\xi_{j}})\geq\delta$.
It follows from \eqref{3.5} and \eqref{1.2} that
$$\inf_{i\neq j}\|\widetilde X_{i}-\widetilde X_{j}\|^{2}_2=\inf_{i\neq j}\Bbb{E}(\widetilde X_{i}-\widetilde X_{j})^{2}
\geq
2-\frac{2e(\tilde{\boldsymbol\xi_{i}},\tilde{\boldsymbol\xi_{j}},n)}{c_{1}(d)n^{d}}\geq
1.$$
Thus, by \eqref{2.8} and Lemma \ref{lem3.2}, we have
\begin{align}\label{3.7}
N^{\rm{ave}}_{2,\infty}(\Pi^{d}_{n})&\asymp n^{-d/2}\Bbb
E\|P_{\a}\|_\infty\asymp n^{-d/2}\Bbb E \max_{1\le j\le
M}|P_{\a}({\boldsymbol\xi_{j}})|\asymp\Bbb E\max_{1\le j\le
M}|X_j|\notag\\ &\ge \Bbb E\max_{1\le j\le \widetilde
M}|\widetilde X_j| \gg\sqrt{\ln \widetilde M}\inf_{i\neq
j}\|\widetilde X_{i}-\widetilde X_{j}\|_2\geq \sqrt{\ln \widetilde
M}\asymp \sqrt{\ln n} .
\end{align}

A combination of \eqref{3.4} and \eqref{3.7} yields
\begin{equation}\label{3.8}
N^{\rm{ave}}_{2,\infty}(\Pi^{d}_{n})\asymp\Bbb E \max_{1\le j\le M}|X_j|\asymp \sqrt{\ln
n}.\end{equation}

Theorem \ref{thm3.1} is proved.
 \end{proof}

Next, we want to get the upper bound of $\mathbb{E}\|P_{\a}\|^{-r}_{\infty}$. We need the following lemma.
\begin{lem}\label{lem3.4} {\rm (\cite[Theorem 4]{LO},\, \cite[Theorem 3.3]{PV}.)} Let $X=\{X_j\}^{M}_{j=1}\in\mathbb{R}^{M}$ is a centered Gaussian
random vector  and $X_{j}\sim N(0,1),\  1\leq j\leq M$. Then, for
all $t\in(0,1/2)$, we have
\begin{equation*}
\mathbb{P}(\max_{1\leq j\leq M}|X_{j}|\leq tm_{X})\leq
\frac{1}{2}e^{-\frac{1}{4}m^{2}_{X}\ln(\frac{1}{2t})},
\end{equation*}
where $m_{X}$ denotes the median of $\max\limits_{1\leq j\leq
M}|X_{j}|$, that is
$$\mathbb{P}(\max\limits_{1\leq j\leq M}|X_{j}|<
m_{X})\leq\frac{1}{2}\ \ { and}\ \
\mathbb{P}(\max\limits_{1\leq j\leq M}|X_{j}|>
m_{X})\leq\frac{1}{2}.$$
In particular, if $m_{X}\geq
c_{0}\sqrt{\ln n}$, then for all $t\in(0,1/2)$, we have
\begin{equation}\label{3.9}
\mathbb{P}(\max_{1\leq j\leq M}|X_{j}|\leq c_{0}t\sqrt{\ln n})\leq
\frac{1}{2}(2t)^{\frac{1}{4}c_{0}^{2}\ln n}.
\end{equation}
\end{lem}

\begin{thm}\label{thm3.5}
Let $r>0$. Then for $\ln n>\frac{4r+4}{c^{2}_{0}}$ and $N>r$, we have
\begin{equation}\label{3.10}
\mathbb{E}\|P_{\a}\|^{-r}_{\infty}\ll n^{-rd/2}(\ln
n)^{-r/2}.
\end{equation}
\end{thm}
\begin{proof} As in the proof of Theorem \ref{thm3.1},
we set
$$X=\{X_j\}^{M}_{j=1},\ \ X_j=\frac{P_{\a}({\boldsymbol\xi_j})}{\sqrt{e({\boldsymbol\xi_{j}},{\boldsymbol\xi_{j}},n)}}, \ 1\le j\le
 M,$$
where $\Xi=\{{\boldsymbol\xi_{1}},\ldots,{\boldsymbol\xi_{M}}\}$ is a maximal  $\frac{\delta_0}n$-separated subset of $\mathbb{M}$. Then $X$ is a centered Gaussian random vector and $X_{j}\sim N(0,1),1\leq j\leq M$. Therefore,
\begin{align*}
\mathbb{E}\|P_{\a}\|^{-r}_{\infty}
&\leq\mathbb{E}\Big(\max\limits_{1\leq j\leq
M}|P_{\a}({\boldsymbol\xi_{j}})|\Big)^{-r}\asymp n^{-\frac{dr}{2}}\mathbb{E}\Big(\max\limits_{1\leq
j\leq M} |X_{j}|\Big)^{-r}.
\end{align*}
By \cite[Corollary 3.2]{LT} and (\ref{3.8}), we have
$$m_{X}\asymp\mathbb{E}\max_{1\leq j\leq M}|X_{j}|\asymp\sqrt{\ln
n},$$
which implies that
$$m_{X}\geq c_{0}\sqrt{\ln n},$$
where $m_X$ is the median of $\max\limits_{1\leq j\leq
M}|X_{j}|$, and $c_0>0$ is a constant depending only on $d$. By \eqref{3.9} with $s=c_{0}t\sqrt{\ln n}$, we get for any
$s\in(0,\frac{c_{0}}{2}\sqrt{\ln n})$
\begin{equation}\label{3.11}
\mathbb{P}\Big(\max\limits_{1\leq j\leq M}|X_{j}|\leq
s\Big)\leq\frac{1}{2}\Big(\frac{2s}{c_0\sqrt{\ln
n}}\Big)^{\frac{1}{4}c_0^{2}\ln n}.
\end{equation}

By the variable substitution $t=s^{-1}$, we have
\begin{align*}
&\mathbb{E}\Big(\max\limits_{1\leq j\leq M}|X_{j}|\Big)^{-r}
=r\int^{\infty}_{0}t^{r-1}\mathbb{P}\Big(\Big(\max\limits_{1\leq
j\leq M}
|X_{j}|\Big)^{-1}> t\Big)dt
=r\int^{\infty}_{0}t^{r-1}\mathbb{P}\Big(\max\limits_{1\leq j\leq M}
|X_{j}|<\frac{1}{t}\Big)dt\\
&\le r\int^{\infty}_{0}\mathbb{P}\Big(\max\limits_{1\leq j\leq M}|X_{j}|\leq s\Big)\frac{ds}{s^{r+1}}
\asymp\Big(\int^{\frac{c_{0}}{2}\sqrt{\ln n}}_{0}+\int^{\infty}_{\frac{c_{0}}{2}\sqrt{\ln n}}
\Big)\mathbb{P}\Big(\max\limits_{1\leq j\leq M}|X_{j}|\leq s\Big)\frac{ds}{s^{r+1}}
=I+II,
\end{align*}
where $c_0$ is the constant given in \eqref{3.11}.

We note that
\begin{equation}\label{3.12}
\frac{1}{s^{r+1}}\mathbb{P}\Big(\max\limits_{1\leq j\leq
M}|X_{j}|\leq s\Big)\leq
\frac{1}{2s^{r+1}}\Big(\frac{2s}{c_0\sqrt{\ln
n}}\Big)^{\frac{1}{4}c_0^{2}\ln n},
\end{equation}
and
\begin{equation}\label{3.13}
\mathbb{P}\Big(\max\limits_{1\leq j\leq M}|X_{j}|\leq s\Big)\leq1.
\end{equation}
Using the variable substitution
$s=\frac{c_{0}\sqrt{\ln n}}{2}t$, by \eqref{3.12} we obtain
\begin{align*}
I&=\int^{\frac{c_{0}}{2}\sqrt{\ln n}}_{0}\mathbb{P}\Big(\max\limits_{1\leq j\leq M}|X_{j}|\leq s\Big)\frac{1}{s^{r+1}}ds\ll\int^{\frac{c_{0}}{2}\sqrt{\ln n}}_{0}\frac{1}{s^{r+1}}\Big(\frac{2s}{c_{0}\sqrt{\ln n}}\Big)^{\frac{1}{4}c^{2}_{0}\ln n}ds\\
&=\frac{2^{r}}{(c_{0}\sqrt{\ln n})^{r}}\int^{1}_{0}t^{\frac{1}{4}c^{2}_{0}\ln n-r-1}dt=\frac{2^{r}}{(c_{0}\sqrt{\ln n})^{r}(\frac{1}{4}c^{2}_{0}\ln n-r)}\ll(\ln n)^{-\frac{r}{2}}.
\end{align*}
In view of \eqref{3.13},
\begin{align*}
II&=\int^{\infty}_{c_{0}(\ln n)^{1/2}}
\mathbb{P}\Big(\max\limits_{1\leq j\leq M}|X_{j}|\leq
s\Big)\frac{1}{s^{r+1}}ds\leq\int^{\infty}_{c_{0}(\ln
n)^{1/2}}\frac{1}{s^{r+1}}ds\asymp(\ln n)^{-\frac{r}{2}}.
\end{align*}
Hence, we obtain
\begin{align*}
\mathbb{E}\Big(\max\limits_{1\leq k\leq M}
|X_{k}|\Big)^{-r}&\ll(\ln n)^{-\frac{r}{2}}.
\end{align*}
This deduces that
$$\mathbb{E}\|P_{\a}\|^{-r}_{\infty}\ll
n^{-\frac{dr}{2}}(\ln n)^{-\frac{r}{2}},$$
which completes the proof of Theorem \ref{thm3.5}.
\end{proof}

\section{proof of Theorem \ref{thm1.1}}\label{section4}

\begin{proof} We first prove the upper bounds.

Let $1\le p< q\le \infty$.
 By the  H\"{o}lder inequality, we have
$$\|P_{\a}\|^{2}_{2}\leq\|P_{\a}\|^{\frac{1}{2}}_{1}
\|P_{\a}\|^{\frac{3}{2}}_{3},$$
which implies that
$$\mathbb{E}\frac{\|P_{\a}\|_{q}}{\|P_{\a}\|_{1}}\leq\mathbb{E}
\bigg(\frac{\|P_{\a}\|_{q}\|P_{\a}\|^{3}_{3}}{\|P_{\a}\|^{4}_{2}}\bigg)
=\mathbb{E}\bigg(\frac{\|P_{\a}\|_{q}}{\|P_{\a}\|^{2}_{2}}
\frac{\|P_{\a}\|^{3}_{3}}{\|P_{\a}\|^{2}_{2}}\bigg)
\leq\bigg(\mathbb{E}\frac{\|P_{\a}\|^{2}_{q}}
{\|P_{\a}\|^{4}_{2}}
\bigg)^{1/2}\bigg(\mathbb{E}\frac{\|P_{\a}\|^{6}_{3}}
{\|P_{\a}\|^{4}_{2}}\bigg)^{1/2}.$$
By (\ref{2.6}) we deduce that
$$\mathbb{E}\frac{\|P_{\a}\|^{2}_{q}}{\|P_{\a}\|^{4}_{2}}\asymp
n^{-2d}\mathbb{E}\|P_{\a}\|^{2}_{q},$$
and
$$\mathbb{E}\frac{\|P_{\a}\|^{6}_{3}}{\|P_{\a}\|^{4}_{2}}
\asymp n^{-2d}\mathbb{E}\|P_{\a}\|^{6}_{3}.$$
According to  \cite[Theorem 1]{WZ} and (\ref{3.2}), we have for
$s\ge1$,
\begin{equation}\label{4.1}\mathbb{E}\|P_{\a}\|_{q}^{s}\asymp \left\{
\begin{aligned}
 &n^{ds/2},&&\ \,1\leq q<\infty,\\
 &n^{ds/2}(\ln n)^{s/2},&&\ \, q=\infty.
\end{aligned}
\right.
\end{equation}
It follows  that
\begin{equation*}
\mathbb{E}\frac{\|P_{\a}\|_{q}}{\|P_{\a}\|_{p}}\leq\mathbb{E}\frac{\|P_{\a}\|_{q}}{\|P_{\a}\|_{1}}\ll\left\{
\begin{aligned}
 &1,&\ \,1\leq p\leq q<\infty,\\
 &\sqrt{\ln n},&\ \,1\leq p<q=\infty.
\end{aligned}
\right.
\end{equation*}

Let $1\leq q< p\leq\infty$.  For $p<\infty$,  we have
$$\mathbb{E}\frac{\|P_{\a}\|_{q}}{\|P_{\a}\|_{p}}\ll 1.$$
For $p=\infty$, by the Cauchy inequality we have
\begin{align*}
\mathbb{E}\frac{\|P_{\a}\|_{q}}{\|P_{\a}\|_{\infty}}&
=\mathbb{E}\bigg(\frac{\|P_{\a}\|_{q}}{\|P_{\a}\|_{2}}\frac{\|P_{\a}\|_{2}}{\|P_{\a}\|_{\infty}}\bigg)
\leq\bigg(\mathbb{E}\frac{\|P_{\a}\|^{2}_{q}}{\|P_{\a}\|^{2}_{2}}\bigg)^{1/2}
\bigg(\mathbb{E}\frac{\|P_{\a}\|^{2}_{2}}{\|P_{\a}\|^{2}_{\infty}}\bigg)^{1/2}.
\end{align*}
It follows from (\ref{2.7}) and (\ref{3.10}) that
\begin{equation}\label{4.2}
\mathbb{E}\frac{\|P_{\a}\|^{2}_{2}}{\|P_{\a}\|^{2}_{\infty}}\asymp
n^{d} \, \mathbb{E}\|P_{\a}\|^{-2}_{\infty}\ll (\ln n)^{-1}.
\end{equation}
In view of (\ref{2.8}) and \eqref{4.1}
\begin{equation*}
\mathbb{E}\frac{\|P_{\a}\|^{2}_{q}}{\|P_{\a}\|^{2}_{2}}\asymp
n^{-d}\,\mathbb{E}\|P_{\a}\|^{2}_{q}\asymp1,
\end{equation*}
which, together with (\ref{4.2}), yields that
$$\mathbb{E}\frac{\|P_{\a}\|_{q}}{\|P_{\a}\|_{\infty}}\ll
\frac{1}{\sqrt{\ln n}}.$$

Hence, we obtain the upper bounds of
$N^{\rm{ave}}_{p,q}(\mathcal{P}_{n})$ for $1\le p,q\le \infty$ as
follows.
\begin{equation}\label{4.3}
N^{\rm{ave}}_{p,q}(\mathcal{P}_{n})\ll\left\{
\begin{aligned}
 &1,&&\ \,1\leq p,q<\infty\ {\rm or}\ p=q=\infty,\\
 &(\ln n)^{1/2},&&\ \,1\leq p<q=\infty,\\
 &(\ln n)^{-1/2},&&\ \,1\leq q<p=\infty.
\end{aligned}
\right.
\end{equation}

Next, we prove the lower bounds. By the Cauchy inequality we have
\begin{align*}
1=\mathbb{E}(1)&=\mathbb{E}\bigg(\frac{\|P_{\a}\|_{q}}{\|P_{\a}\|_{p}}
\frac{\|P_{\a}\|_{p}}{\|P_{\a}\|_{q}}\bigg)^{1/2}\leq\bigg(\mathbb{E}\frac{\|P_{\a}\|_{q}}{\|P_{\a}\|_{p}}\bigg)^{1/2}
\bigg(\mathbb{E}\frac{\|P_{\a}\|_{p}}{\|P_{\a}\|_{q}}\bigg)^{1/2},
\end{align*}
which deduces that
\begin{equation}\label{4.4}
\mathbb{E}\frac{\|P_{\a}\|_{q}}{\|P_{\a}\|_{p}}\geq
\bigg(\mathbb{E}\frac{\|P_{\a}\|_{p}}{\|P_{\a}\|_{q}}\bigg)^{-1}.
\end{equation}
By \eqref{4.3} and \eqref{4.4} we get
\begin{equation}\label{4.5}
N^{\rm{ave}}_{p,q}(\mathcal{P}_{n})\gg \left\{
\begin{aligned}
 &1,&&\ \,1\leq p,q<\infty\ {\rm or}\ p=q=\infty,\\
 &(\ln n)^{1/2},&&\ \,1\leq p<q=\infty,\\
 &(\ln n)^{-1/2},&&\ \,1\leq q<p=\infty,
\end{aligned}
\right.
\end{equation}
which gives the lower bounds of
$N^{\rm{ave}}_{p,q}(\mathcal{P}_{n})$ for $1\le p,q\le \infty$.

Finally, combining \eqref{4.3} and \eqref{4.5}, we complete the the proof of Theorem \ref{thm1.1}.\end{proof}

\

\noindent\textbf{Acknowledgments}
  The authors  were
supported by the National Natural Science Foundation of China
(Project no. 12371098).

\end{document}